\documentclass[10pt]{amsart}%
\usepackage{amsmath}
\usepackage{amssymb}
\usepackage{amsfonts}
\usepackage{graphicx}%
\newtheorem{theorem}{Theorem}[section]
\newtheorem{lemma}{Lemma}[section]
\theoremstyle{definition}

\theoremstyle{remark}
\newtheorem{remark}{Remark}[section]
\numberwithin{equation}{section}
\theoremstyle{plain}

\newtheorem{proposition}{Proposition}[section]
\newtheorem{example}{Example}[section]

\begin{document}
\title[Exceptional Rare Events]{\textbf{Hitting-time Limits for some Exceptional Rare Events of Ergodic Maps}\\\vspace{0.2cm}}
\author{Roland Zweim\"{u}ller}
\address{Fakult\"{a}t f\"{u}r Mathematik, Universit\"{a}t Wien, Oskar-Morgenstern-Platz
1, 1090 Vienna, Austria }
\email{roland.zweimueller@univie.ac.at}
\urladdr{http://www.mat.univie.ac.at/\symbol{126}zweimueller/ }

\begin{abstract}
We discuss limit distributions for hitting-time functions of certain
\textquotedblleft exceptional\textquotedblright\ families of asymptotically
rare events for ergodic probability preserving transformations. The abstract
core is an inducing argument. The latter applies, for example, to shrinking
intervals around periodic points (both uniformly expanding and neutral) of
certain finite measure preserving interval maps. In particular, we give a
complete answer to a question raised in \cite{FFTV}.

\end{abstract}
\maketitle

\section{Introduction}

Hitting- and return-time statistics for asymptotically rare events in ergodic
dynamical systems have been the subject of intense research over the last
years. For a wide variety of probability preserving systems with some
hyperbolicity it has been shown that for natural families of rare events, like
cylinders or $\varepsilon$-balls shrinking to a given distinguished point
$x^{\ast}$, convergence (after normalization by the measure of the sets) to a
standard exponential random variable $\mathcal{E}$ (with $\Pr[\mathcal{E}%
>t]=e^{-t}$ for $t\geq0$) is typical in that it holds for almost every point
$x^{\ast}$. Nonetheless, there are often exceptional points at which a
different asymptotic behaviour is observed. It is not hard to understand that
this should be so in the case of periodic points, when a definite proportion
of the set returns after a fixed number of steps, thus giving rise to a point
mass at the origin for the limit of scaled return-times. (For some basic
classes of systems dichotomy results have been established which confirm that
there are no other exceptional points.) There has also been some interest in
neutral repellers, which lead to a trivial limit under the usual
normalization, but may in fact give rise to a nice limit law when a different
scale is used.

The purpose of the present note is to communicate an abstract inducing
argument which can be used to clarify the asymptotics of such exceptional rare
events once it is known that the standard exponential limit arises in certain
situations. It can lead to straightforward and quick proofs. This will be
illustrated in the setup of simple prototypical piecewise invertible expanding
interval maps. We focus on the basic case of fixed points. Extending the
arguments to periodic points only requires routine arguments, and hardly gives
new insights.

First, we show that the (well-known) exceptional hitting-time limit
$\theta^{-1}\mathcal{E}$ with expectation $\theta^{-1}>1$ at a repelling
hyperbolic fixed point of a well-behaved map can be obtained using this
approach. Next, we demonstrate that the term \textquotedblleft
well-behaved\textquotedblright\ in the previous sentence is there for a
reason. We construct a map which looks nice enough\ (a uniformly expanding
piecewise affine Markov map) but nonetheless admits a hyperbolic fixed point
at which the limit variable is a standard exponential $\mathcal{E}$ and not
the exceptional $\theta^{-1}\mathcal{E}$ suggested by the first scenario.

Finally, we turn to the main application of our abstract inducing principle
and consider probability preserving maps with neutral fixed points. Here we
answer a question raised in \cite{FFTV} by clarifying the asymptotic
hitting-time behaviour of neighbourhoods of an indifferent fixed point (with
rather general local behaviour).\newline

\noindent\textbf{Acknowledgement.} R.Z. thanks Gerhard Keller and Sebastian
Fischer for conversations related to this topic, and the referee for carefully
reading the manuscript. This research has benefitted from support by
\emph{\"{O}AD grant 92\"{o}u6}, and from participation at the thematic
semester on \emph{Dynamics and Geometry} at the \emph{Centre Henri
Lebesgue}.\newline

\section{Preparations}

\noindent\textbf{General setup.} Throughout the paper, \emph{all measures are
understood to be finite}. We study (possibly non-invertible) \emph{measure
preserving transformations} $T$ on $(X,\mathcal{A},\mu)$, i.e. measurable maps
$T:X\rightarrow X$ for which $\mu\circ T^{-1}=\mu$. The transformation $T$
will also be \emph{ergodic}. For such a system $(X,\mathcal{A},\mu,T)$, and
any $Y\in\mathcal{A}$ with $\mu(Y)>0$, define the \emph{first hitting time
}function of $Y$, $\varphi_{Y}:X\rightarrow\mathbb{N}\cup\{\infty\}$ by
$\varphi_{Y}(x):=\min\{n\geq1:T^{n}x\in Y\} $, $x\in X$, which is finite a.e.
by ergodicity and the Poincar\'{e} recurrence theorem. Set $T_{Y}%
x:=T^{\varphi(x)}x$, $x\in X$. When restricted to $Y$, $\varphi_{Y}$ is called
the \emph{first return time }of $Y$. If we let $Y\cap\mathcal{A}:=\{Y\cap
A:A\in\mathcal{A}\}$ denote the trace of $\mathcal{A}$ in $Y$, then $\mu
\mid_{Y\cap\mathcal{A}}$ is invariant under the \emph{first return map},
$T_{Y}\,$restricted to $Y$. It is natural to regard $\varphi_{Y}$ as a random
variable on the probability space $(X,\mathcal{A},\mu_{Y})$, where $\mu
_{Y}(E):=\mu(Y)^{-1}\mu(Y\cap E)$. By Kac' formula, it has expectation
$\int\varphi_{Y}\,d\mu_{Y}=\mu(Y)^{-1}$.

In this setup, a sequence $(E_{k})_{k\geq1}$ in $\mathcal{A}$ with $\mu
(E_{k})>0$ and $\mu(E_{k})\rightarrow0$ will be referred to as a sequence of
\emph{asymptotically rare events}. Asking for an \emph{asymptotic hitting-time
distribution} or \emph{hitting-time statistics} (\emph{HTS}) means to look for
normalizing constants $\gamma_{k}>0$ and a nontrivial random variable
$\mathsf{R}$ taking values in $[0,\infty]$ such that
\begin{equation}
\mu(\gamma_{k}\cdot\varphi_{E_{k}}\leq t)\Longrightarrow\Pr[\mathsf{R}\leq
t]\text{ \quad as }k\rightarrow\infty\text{.} \label{Eq_HITTLIM}%
\end{equation}
(Here, of course, the symbol $\Longrightarrow$ means that convergence takes
place at continuity points $t$ of the respective limit distribution function).
By Kac' formula, a canonical candidate for $\gamma_{k}$ is given by $\mu
(E_{k})$.\vspace{0.3cm}

\noindent\textbf{Change of measure.} It is a fact, both interesting by itself
and useful as a technical tool, that the convergence (\ref{Eq_HITTLIM})
automatically carries over to all probabilities $\nu\ll\mu$. Given a sequence
$(R_{k})_{k\geq1}$ of measurable functions and a measure $\nu$ on
$(X,\mathcal{A})$, we write
\begin{equation}
R_{k}\overset{\nu}{\Longrightarrow}\mathsf{R}%
\end{equation}
to indicate convergence in law of the $R_{k}$, viewed as random variables on
the on the probability space $(X,\mathcal{A},\nu)$, to a variable $\mathsf{R}
$. For instance, (\ref{Eq_HITTLIM}) can then be expressed as $\gamma_{k}%
\cdot\varphi_{E_{k}}\overset{\mu}{\Longrightarrow}\mathsf{R}$. Corollary 5 of
\cite{Z2} contains the following

\begin{theorem}
[\textbf{Strong distributional convergence of hitting times; } \cite{Z2}%
]\label{T_StrongDistrCgeHittingTimes}Let $(X,\mathcal{A},\mu,T)$ be an ergodic
probability-preserving system, $(E_{k})_{k\geq1}$ a sequence of asymptotically
rare events, $(\gamma_{k})$ a sequence in $(0,\infty)$ with $\gamma
_{k}\rightarrow0$, and $\mathsf{R}$ any random variable with values in
$[0,\infty]$. Then
\begin{equation}
\gamma_{k}\cdot\varphi_{E_{k}}\overset{\nu}{\Longrightarrow}\mathsf{R}\text{
\quad as }k\rightarrow\infty
\end{equation}
holds for \emph{one} probability measure $\nu\ll\mu$ iff it holds for
\emph{all} probabilities $\nu\ll\mu$.
\end{theorem}

\vspace{0.2cm}

We record another basic observation regarding changes of measure. Let
$\overset{\nu}{\longrightarrow}$ denote convergence in measure with respect to
$\nu$.

\begin{proposition}
[\textbf{Characterizing asymptotically rare sequences}]%
\label{P_MeasureByMeasure}Suppose that $(X,\mathcal{A},\mu,T)$ is an ergodic
probability-preserving system and $(E_{k})_{k\geq1}$ a sequence in
$\mathcal{A}$. Then the following are equivalent:\vspace{0.1cm}

\begin{enumerate}
\item[(a)] $\mu(E_{k})\rightarrow0$ as $k\rightarrow\infty$,\vspace{0.15cm}

\item[(b)] $\nu(E_{k})\rightarrow0$ as $k\rightarrow\infty$ for \emph{all}
probabilities $\nu\ll\mu$,\vspace{0.1cm}

\item[(c)] $\varphi_{E_{k}}\overset{\mu}{\longrightarrow}\infty$ as
$k\rightarrow\infty$,\vspace{0.1cm}

\item[(d)] $\varphi_{E_{k}}\overset{\nu}{\longrightarrow}\infty$ as
$k\rightarrow\infty$\ for \emph{all} probabilities $\nu\ll\mu$,\vspace{0.1cm}

\item[(e)] $\varphi_{E_{k}}\overset{\nu}{\longrightarrow}\infty$ as
$k\rightarrow\infty$\ for \emph{some} probability $\nu\ll\mu$.
\end{enumerate}
\end{proposition}

\begin{proof}
\textbf{(i)} Equivalence of (a) and (b) is immediate from the
\textquotedblleft continuity\textquotedblright\ characterization of absolute
continuity, and so is equivalence of (c) and (d).

To check that (b) entails (d), fix $\nu\ll\mu$ and some $N\geq1$. Then
$\nu(\varphi_{E_{k}}\leq N)=\nu({\textstyle\bigcup_{n=1}^{N}} T^{-n}E_{k}%
)\leq\sum_{n=1}^{N}\nu(T^{-n}E_{k})$. But for every $n\geq1$, we have
$\nu(T^{-n}E_{k})\rightarrow0$ as $k\rightarrow\infty$ by (b) because $\nu
\ll\mu$. Hence (d) follows as $\nu(\varphi_{E_{k}}\leq N)\rightarrow0$.
\newline\newline\textbf{(ii)} To establish the more interesting fact that (e)
implies (b), take $\nu$ as in (e) and choose another probability
$\widetilde{\nu}\ll\mu$. To prove $\widetilde{\nu}(E_{k})\rightarrow0$ assume
the contrary, meaning that there are $\delta>0$ and $k_{j}\nearrow\infty$ such
that $\widetilde{\nu}(E_{k_{j}})\geq\delta$ for all $j\geq1$. We show that
this contradicts $\varphi_{E_{k}}\overset{\nu}{\longrightarrow}\infty$.

Set $u:=d\nu/d\mu$ and $\widetilde{u}:=d\widetilde{\nu}/d\mu$. We first
consider the case where $\left\Vert \widetilde{u}\right\Vert _{\infty}<\infty
$. Let $\widehat{T}$ denote the transfer operator of $T$ with respect to $\mu
$, so that $\int(f\circ T)\,g\,d\mu=\int f\,\widehat{T}g\,d\mu$ whenever $f\in
L_{\infty}(\mu)$ and $g\in L_{1}(\mu)$. Since $T$ is ergodic and recurrent, we
have $\sum_{n\geq1}\widehat{T}^{n}u=\infty$ $\mu$-a.e. on $X$. Therefore we
can choose (and fix) some $N\geq1$ such that $F:=\{{\textstyle\sum
\nolimits_{n=1}^{N}}\widehat{T}^{n}u\geq1\}$ satisfies $\widetilde{\nu}%
(F^{c})<\delta/2$. Then,
\begin{equation}
\widetilde{\nu}\left(  E_{k_{j}}\cap F\right)  \geq\delta/2\text{ \quad for
}j\geq1\text{.} \label{Eq_nbvchfgdsv}%
\end{equation}
Also,
\begin{align}
\widetilde{\nu}\left(  E_{k_{j}}\cap F\right)   &  ={\textstyle\int
}1_{E_{k_{j}}}1_{F}\,\widetilde{u}\,d\mu\\
&  \leq\left\Vert \widetilde{u}\right\Vert _{\infty}{\textstyle\int
}1_{E_{k_{j}}}\left(  {\textstyle\sum\nolimits_{n=1}^{N}}\widehat{T}%
^{n}u\right)  \,d\mu\nonumber\\
&  =\left\Vert \widetilde{u}\right\Vert _{\infty}{\textstyle\sum
\nolimits_{n=1}^{N}}\nu(T^{-n}E_{k_{j}})\text{.}\nonumber
\end{align}
Combining this with (\ref{Eq_nbvchfgdsv}), we see that for every $j\geq1$
there is some $n_{j}\in\{1,\ldots,N\}$ such that $\nu(T^{-n_{j}}E_{k_{j}}%
)\geq\delta^{\prime}:=\delta/(2N\left\Vert \widetilde{u}\right\Vert _{\infty
})>0$. But $\varphi_{E}\leq N$ on $T^{-n}E $ if $1\leq n\leq N$. Hence we
conclude that $\nu(\varphi_{E_{k_{j}}}\leq N)\geq\delta^{\prime}$ for $j\geq1$.

Finally, in the case of unbounded $\widetilde{u}$, we still have
$\widetilde{\nu}\ll\overline{\nu}$ for some probability $\overline{\nu}\ll\mu$
with a bounded density $\overline{u}:=d\overline{\nu}/d\mu$. By the above,
$\overline{\nu}(E_{k})\rightarrow0$ which easily implies $\widetilde{\nu
}(E_{k})\rightarrow0$ by absolute continuity.
\end{proof}

\vspace{0.2cm}

\begin{remark}
[\textbf{A null-preserving }$\sigma$\textbf{-finite version of the
proposition}]The proof above is formulated in such a way that it actually
shows the following: Let $(X,\mathcal{A},\mu)$ be a $\sigma$-finite measure
space and $T$ a \emph{null-preserving} map (measurable with $\mu\circ
T^{-1}\ll\mu$) which is \emph{conservative} (that is recurrent, $A\subseteq
{\textstyle\bigcup_{n\geq1}} T^{-n}A$ (mod $\mu$) for all $A\in\mathcal{A}$)
and \emph{ergodic} ($A=T^{-1}A\in\mathcal{A}$ implies $0\in\{\mu(A),\mu
(A^{c})\}$). Then, for any sequence $(E_{k})_{k\geq1}$ in $\mathcal{A}$,
statements (b), (d) and (e) above are equivalent.

This extension is of interest in situations with an \emph{infinite} invariant
measure $\mu$, where assertions (a) and (c) are no longer about probabilities.
See \cite{Z8} for asymptotic hitting-time distributions of certain sequences
$(E_{k})$ with $\mu(E_{k})=\infty$ for all $k$, but still satisfying (b), (d)
and (e).
\end{remark}

\vspace{0.3cm}

\noindent\textbf{Hitting times and inducing.} As with various other
assertions, proving a statement like (\ref{Eq_HITTLIM}) is often facilitated
by passing to a suitable (nicer) induced map $T_{Y}$, and studying the same
question for this new system. To this end, let $\varphi_{E}^{Y}:Y\rightarrow
\overline{\mathbb{N}}$ denote the hitting time of $E\in\mathcal{A}\cap Y$
under the first-return map $T_{Y}$, that is,
\begin{equation}
\varphi_{E}^{Y}(x):=\inf\{j\geq1:T_{Y}^{j}x\in E\}\text{,\quad}x\in Y\text{.}%
\end{equation}
Then the following inducing principle for hitting-time limits shows that (in
the standard case $\gamma_{k}=\mu(E_{k})$) it suffices to analyse the
distributions of the $\varphi_{E_{k}}^{Y}$ on $Y$.

\begin{theorem}
[\textbf{Hitting-time statistics via inducing; \cite{HWZ}}]%
\label{T_InduceHittingTimeStats}Let $(X,\mathcal{A},\mu,T)$ be an ergodic
probability-preserving system, and $Y\in\mathcal{A}$, $\mu(Y)>0$. Assume that
$(E_{k})_{k\geq1}$ is a sequence of asymptotically rare events in
$\mathcal{A}\cap Y$, and that $\mathsf{R}$ is any random variable with values
in $[0,\infty]$. Then
\begin{equation}
\,\mu_{Y}(E_{k})\,\varphi_{E_{k}}^{Y}\overset{\mu_{Y}}{\Longrightarrow
}\mathsf{R}\quad\text{as }k\rightarrow\infty\text{ }
\label{Eq_CgeOfInducedHittingTimes}%
\end{equation}
iff
\begin{equation}
\mu(E_{k})\,\varphi_{E_{k}}\overset{\mu}{\Longrightarrow}\mathsf{R}%
\quad\text{as }k\rightarrow\infty\text{.} \label{Eq_CgeOfOriginalHittingTimes}%
\end{equation}

\end{theorem}

A result of this flavour was first discussed in \cite{BruinEtAl} under
additional assumptions. Below we shall provide an even more flexible version,
and use it to study the exceptional situations mentioned in the introduction.

\section{Inducing hitting-time statistics - revisited}

The core of this paper is a very useful extension of Theorem
\ref{T_InduceHittingTimeStats} quoted above. It allows us to sometimes replace
a target set $E$, not necessarily contained in $Y$, by a more convenient set
$E^{\prime}$ inside $Y$. We shall say that \emph{points of }$Y$\emph{\ can
only reach }$E$\emph{\ via }$E^{\prime}$ if for every $n\geq0$ and a.e. $x\in
Y$,
\begin{equation}
T^{n}x\in E\text{ \quad implies that \quad}T^{j}x\in E^{\prime}\text{ for some
}j\in\{0,\ldots,n\}\text{,} \label{Eq_TheConditionClever}%
\end{equation}
that is, orbits starting in $Y$ cannot visit the set $E$ before $E^{\prime}$
is visited\footnote{In contrast to the definition of the first hitting-time,
we also take visits at time zero into account here. We do so in order to
obtain a flexible condition which can easily be applied to certain concrete
situations.}.

\begin{example}
\label{Exple_BasicScenario}\textbf{a)} Given an ergodic probability-preserving
system $(X,\mathcal{A},\mu,T)$, let $Y\in\mathcal{A} $, $\mu(Y)>0$, and
consider $E:=Y^{c}\cap\{\varphi_{Y}\geq i\}$ for some $i\geq1$. It is then
immediate that points of $Y$ can only reach $E$ via $E^{\prime}:=Y\cap
T^{-1}E=Y\cap\{\varphi_{Y}>i\}$.\newline\textbf{b)} More generally, if
$E\subseteq Y^{c}$ satisfies $Y^{c}\cap T^{-1}E\subseteq E$ (mod $\mu$), then
points of $Y$ can only reach $E$ via $E^{\prime}:=Y\cap T^{-1}E$. \newline
\end{example}

In case $E^{\prime}\subseteq E$ it is clear that, starting from $Y$, the first
visits to $E$ and $E^{\prime}$, respectively, must then coincide. We can cover
other interesting scenarios if, more generally, we only require that these
times do not differ too much. Then the hitting time distributions of $E$ and
$E^{\prime}$ will be comparable when the sets are small, as made precise in
the following result.

\begin{theorem}
[\textbf{Hitting-time statistics via inducing; extended version}%
]\label{T_InduceHittingTimeStatsNew}Let $(X,\mathcal{A},\mu,T)$ be an ergodic
probability-preserving system, and $(E_{k})_{k\geq1}$ a sequence of
asymptotically rare events. Suppose that $Y\in\mathcal{A}$ is a set of
positive measure, $(E_{k}^{\prime})_{k\geq1}$ a sequence in $\mathcal{A}\cap
Y$ such that, for every $k\geq1$, points of $Y$ can only reach $E_{k}$ via
$E_{k}^{\prime}$, and let $\mathsf{R}$ be any random variable with values in
$[0,\infty]$.\newline\textbf{a)} If
\begin{equation}
\mu(E_{k}^{\prime})(\varphi_{E_{k}}-\varphi_{E_{k}^{\prime}})\overset{\mu_{Y}%
}{\longrightarrow}0\text{ \quad as }k\rightarrow\infty\text{,}
\label{Eq_ChangeOfTimeSmallEnough}%
\end{equation}
then $(E_{k}^{\prime})$ is asymptotically rare, and
\begin{equation}
\mu_{Y}(E_{k}^{\prime})\,\varphi_{E_{k}^{\prime}}^{Y}\overset{\mu_{Y}%
}{\Longrightarrow}\,\mathsf{R}\text{ \quad as }k\rightarrow\infty\text{,}
\label{Eq_CgeHitForInducedSys}%
\end{equation}
holds iff
\begin{equation}
\mu(E_{k}^{\prime})\,\varphi_{E_{k}}\overset{\mu}{\Longrightarrow}%
\,\mathsf{R}\text{ \quad as }k\rightarrow\infty\text{.}
\label{Eq_CgeHitForOrigSys}%
\end{equation}
\textbf{b)} If there exists some constant $M\geq0$ such that $E_{k}^{\prime
}\subseteq{\textstyle\bigcup\nolimits_{m=0}^{M}}T^{-m}E_{k}$ for $k\geq1$,
then $(E_{k}^{\prime})$ satisfies assumption (\ref{Eq_ChangeOfTimeSmallEnough}%
) of a).
\end{theorem}

\vspace{0.2cm}

This result contains Theorem \ref{T_InduceHittingTimeStats}. (Take
$E_{k}=E_{k}^{\prime}\subseteq Y$, and $M=0$ in part b).)

\begin{remark}
[\textbf{The normalizing constants}]\label{Rem_cfcfcfccf}Note that the
normalizing factor on the left-hand side of (\ref{Eq_CgeHitForOrigSys}) really
is $\mu(E_{k}^{\prime})$, and not $\mu(E_{k})$. This is in fact one main point
of the result. The relation between $\mu(E_{k}^{\prime})$ and $\mu(E_{k})$
determines what really happens to $\varphi_{E_{k}}$. Let $\theta
:=\lim_{k\rightarrow\infty}\mu(E_{k}^{\prime})/\mu(E_{k})$ in case this limit exists.

\textbf{a)} If $\theta>0$, then (\ref{Eq_CgeHitForOrigSys}) is equivalent to
\begin{equation}
\mu(E_{k})\,\varphi_{E_{k}}\overset{\mu}{\Longrightarrow}\,\theta
^{-1}\mathsf{R}\text{ \quad as }k\rightarrow\infty\text{.}%
\end{equation}
This is what typically happens at repelling periodic points, see Theorem
\ref{T_TheWellKnownFolkloreFixedPoint} below.

\textbf{b)}\ There are also interesting situations in which $\theta=0$. In
this case the theorem identifies a possibly non-trivial limit variable
$\mathsf{R}$ for the hitting-times $\varphi_{E_{k}}$ on a scale $\mu
(E_{k}^{\prime})$ essentially different from the canonical scale $\mu(E_{k})$,
thus identifying the \textquotedblleft correct scale\textquotedblright\ for
these variables. Indeed, this is what happens in the case of neighbourhoods
$E_{k}$ of indifferent fixed points for probability-preserving intermittent
maps, see Theorem \ref{T_HTSatNeutralFP} below.

\textbf{c)} For a sequence of pairs $(E_{k},E_{k}^{\prime})$ which satisfy
$E_{k}^{\prime}=Y\cap T^{-1}E_{k}$ (as in Example \ref{Exple_BasicScenario})
b), we have $\theta=\lim_{k\rightarrow\infty}(1-\mu(Y^{c}\cap T^{-1}E_{k}%
)/\mu(E_{k}))$.
\end{remark}

\vspace{0.3cm}

\begin{proof}
[\textbf{Proof of Theorem \ref{T_InduceHittingTimeStatsNew}.}]\textbf{(i)} We
first show that $(E_{k}^{\prime})$ is asymptotically rare. Assume otherwise,
then there are $\delta>0$ and $k_{j}\nearrow\infty$ such that $\mu(E_{k_{j}%
}^{\prime})\geq\delta$ for $j\geq1$. Since $\mid\varphi_{E_{k}}-\varphi
_{E_{k}^{\prime}}\mid\geq1$ on $\{\varphi_{E_{k}}\neq\varphi_{E_{k}^{\prime}%
}\}$, (\ref{Eq_ChangeOfTimeSmallEnough}) ensures that
\begin{equation}
\mu_{Y}(\varphi_{E_{k_{j}}}\neq\varphi_{E_{k_{j}}^{\prime}})\longrightarrow
0\text{ \quad as }j\rightarrow\infty\text{.} \label{Eq_bcxvnvynvcybzzzzzzz}%
\end{equation}
By assumption $(E_{k})$ is asymptotically rare and Proposition
\ref{P_MeasureByMeasure} gives $\varphi_{E_{k}}\overset{\mu_{Y}}%
{\longrightarrow}\infty$. Due to (\ref{Eq_bcxvnvynvcybzzzzzzz}) we then get
$\varphi_{E_{k_{j}}^{\prime}}\overset{\mu_{Y}}{\longrightarrow}\infty$ as
well. Hence $\mu(E_{k_{j}}^{\prime})\rightarrow0$ by another application of
the proposition.\newline\newline\textbf{(ii)} By Theorem
\ref{T_InduceHittingTimeStats}, the convergence in
(\ref{Eq_CgeHitForInducedSys}) is equivalent to $\mu(E_{k}^{\prime}%
)\,\varphi_{E_{k}^{\prime}}\overset{\mu}{\Longrightarrow}\,\mathsf{R}$. Due to
Theorem \ref{T_StrongDistrCgeHittingTimes} this, in turn, is equivalent to
\begin{equation}
\mu(E_{k}^{\prime})\,\varphi_{E_{k}^{\prime}}\overset{\mu_{Y}}{\Longrightarrow
}\,\mathsf{R}\text{ \quad as }k\rightarrow\infty\text{.}
\label{Eq_sjdhgfkjagbfvka}%
\end{equation}
As a consequence of (\ref{Eq_ChangeOfTimeSmallEnough}), we see that
(\ref{Eq_sjdhgfkjagbfvka}) is equivalent to $\mu(E_{k}^{\prime})\,\varphi
_{E_{k}}\overset{\mu_{Y}}{\Longrightarrow}\,\mathsf{R}$, and in view of
Theorem \ref{T_StrongDistrCgeHittingTimes}, the latter is indeed the same as
(\ref{Eq_CgeHitForOrigSys}).\newline\newline\textbf{(iii)} As $T$ preserves
$\mu$, $E_{k}^{\prime}\subseteq{\textstyle\bigcup\nolimits_{m=0}^{M}}%
T^{-m}E_{k}$ entails $\mu(E_{k}^{\prime})\leq(M+1)\mu(E_{k})\rightarrow0$, so
that $(E_{k}^{\prime})$ is asymptotically rare. Next,
\begin{equation}
\varphi_{E_{k}^{\prime}}\leq\varphi_{E_{k}}\leq\varphi_{E_{k}^{\prime}%
}+M\text{ \quad a.e. on }Y\setminus(E_{k}\cup E_{k}^{\prime})\text{,}%
\end{equation}
since points of $Y$ can reach $E_{k}$ only via $E_{k}^{\prime}$, and
$E_{k}^{\prime}\subseteq{\textstyle\bigcup\nolimits_{m=0}^{M}}T^{-m}E_{k}$.
(We discard $Y\cap(E_{k}\cup E_{k}^{\prime})$ because
(\ref{Eq_TheConditionClever}) allows $0\in\{j,n\}$ in which case the time of
the first visit is not the first hitting time.) Therefore, $0\leq\mu
(E_{k}^{\prime})(\varphi_{E_{k}}-\varphi_{E_{k}^{\prime}})\leq M\mu
(E_{k}^{\prime})\rightarrow0$ on $Y\setminus(E_{k}\cup E_{k}^{\prime})$. But
$\mu_{Y}(Y\cap(E_{k}\cup E_{k}^{\prime}))\rightarrow0$ since both $(E_{k})$
and $(E_{k}^{\prime})$ are asymptotically rare. This implies
(\ref{Eq_ChangeOfTimeSmallEnough}).
\end{proof}

\vspace{0.3cm}

\section{Application to uniformly expanding interval maps}

\noindent\textbf{Piecewise monotone interval maps.} A \emph{piecewise
monotonic system} is a triple $(X,T,\xi)$, where $X$ is a bounded interval,
$\xi$ is a collection of nonempty pairwise disjoint open subintervals $Z$ of
$X$ with $\lambda(X\setminus{\textstyle\bigcup\nolimits_{Z\in\xi}} Z)=0$
(where $\lambda$ denotes Lebesgue measure), and $T:X\longrightarrow X$ is such
that each \emph{branch} of $T$, i.e. its restriction to any of its
\emph{cylinders} $Z\in\xi$ is a homeomorphism onto $TZ$. The system is
\emph{Markov} if $TZ\cap Z^{\prime}\neq\varnothing$ for $Z,Z^{\prime}\in\xi$
implies $Z^{\prime}\subseteq TZ$, and \emph{piecewise onto} if $TZ=X$ mod
$\lambda\ $for all $Z\in\xi$.

We focus on systems with $\mathcal{C}^{2}$ branches, and call such a system
$(X,T,\xi)$ a \emph{Folklore map} if it is piecewise onto, uniformly expanding
($\inf\left\vert T^{\prime}\right\vert >1$), and satisfies \emph{Adler's
condition}, meaning that $T^{\prime\prime}/(T^{\prime})^{2}$ \ is bounded. It
is well known that every Folklore map has a unique absolutely continuous
invariant probability measure $\mu$ the density of which admits a continuous
version $h$ bounded away from $0$ and $\infty$.\vspace{0.3cm}

\noindent\textbf{A common scenario with standard exponential limit.} We
briefly record an auxiliary observation about a type of asymptotically rare
events which often arise through inducing. The following is just an easy
variant of well-known results.

\begin{lemma}
[\textbf{Exponential limit for sets containing cylinders}]%
\label{T_MyCompactnessForExpoLimit}Let $(X,T,\xi)$ be a Folklore map and
$\mu\ll\lambda$ its invariant probability. Let $(E_{k}^{\prime})_{k\geq1} $ be
a sequence of intervals shrinking to an endpoint of $X$, and such that
$E_{k}^{\prime}=E_{k}^{\blacktriangle}\cup E_{k}^{\triangle}$ (mod $\lambda$)
with $E_{k}^{\blacktriangle}={\textstyle\bigcup\nolimits_{Z\in\xi:Z\subseteq
E_{k}}} Z$ a union of cylinders, and $E_{k}^{\triangle}$ a (possibly empty)
interval contained in a cylinder $F_{k}\in\xi$. If $\lambda(F_{k}%
)=O(\lambda(E_{k}^{\blacktriangle}))$ as $k\rightarrow\infty$, then
\begin{equation}
\mu(E_{k}^{\prime})\,\varphi_{E_{k}^{\prime}}\overset{\mu}{\Longrightarrow
}\mathcal{E}\text{ \quad as }k\rightarrow\infty\text{.}
\label{Eq_ExpoLimitInMyExpoLimitThm}%
\end{equation}

\end{lemma}

\begin{proof}
This follows by routine arguments using, for example, the ideas of \cite{HSV}.
(The assumptions ensure $\mu_{E_{k}^{\prime}}(\varphi_{E_{k}^{\prime}%
}=1)\rightarrow0$, and that the densities $\widehat{T}(\mu(E_{k}^{\prime
})^{-1}1_{E_{k}^{\prime}})$, $k\geq1$, where $\widehat{T}$ denotes the
transfer operator of $T$ with respect to $\mu$, have uniformly bounded
variation on the unit interval.)

Alternatively, one can apply the perturbation theory of \cite{KL},
\cite{Kell2012} by slightly adapting the argument used for the Gauss map
example in \cite{KL}.
\end{proof}

\vspace{0.1cm}

\begin{remark}
\label{R_UnionOfCyls}Even easier, the same conclusion holds whenever the
$E_{k}^{\prime}$ are sets with $\lambda(E_{k}^{\prime})\rightarrow0$, and such
that each is a union of cylinders. (In this case the argument from \cite{KL}
applies without change.)
\end{remark}

\vspace{0.3cm}

\noindent\textbf{The exceptional behaviour of a Folklore map }$T$\textbf{\ at
a fixed point.} As a warm-up we now show that it is easy to employ Theorem
\ref{T_InduceHittingTimeStatsNew} to determine the hitting-time statistics for
small neighbourhoods of (uniformly repelling) periodic points of Folklore (or
similar) maps. Since this type of result is well known (see e.g.
\cite{Kell2012} or \cite{FFT}) and our emphasis is on the method rather than
the most general version, we focus on the most basic case of a fixed point of
a map with two branches.

\begin{example}
[\textbf{HTS for fixed points of simple Folklore maps}]%
\label{T_TheWellKnownFolkloreFixedPoint}Let $(X,T,\xi)$ be a Folklore map on
$X=[0,1]$ with two increasing branches, $\xi=\{(0,c),(c,1)\}$, and $\mu
\ll\lambda$ its invariant probability. Let $(E_{k})_{k\geq1}$ be a sequence of
intervals which contain the fixed point $x^{\ast}=0$, and such that
$\lambda(E_{k})\rightarrow0$. Then,
\begin{equation}
\mu(E_{k})\,\varphi_{E_{k}}\overset{\mu}{\Longrightarrow}\theta^{-1}%
\mathcal{E}\text{ \quad as }k\rightarrow\infty\text{,}
\label{Eq_TheGoodExceptionalLimit}%
\end{equation}
where $\theta:=1-1/T^{\prime}(0^{+})\in(0,1)$.
\end{example}

\begin{proof}
We use Theorem \ref{T_InduceHittingTimeStatsNew} to quickly derive
(\ref{Eq_TheGoodExceptionalLimit}) from the standard convergence guaranteed in
Lemma \ref{T_MyCompactnessForExpoLimit}: Let $Y:=(c,1)$ be the right-hand
cylinder, and write $E_{k}^{\prime}:=Y\cap T^{-1}E_{k}$. The local dynamics at
$x^{\ast}=0$, together with continuity of the invariant density $h=d\mu
/d\lambda$ give $\mu(Y^{c}\cap T^{-1}E_{k})\sim\mu(E_{k})/T^{\prime}(0^{+})$
as $k\rightarrow\infty$. But as $\mu$ is invariant, we have $\mu(E_{k}%
)=\mu(Y^{c}\cap T^{-1}E_{k})+\mu(E_{k}^{\prime})$. Hence,
\begin{equation}
\mu(E_{k}^{\prime})\sim\theta\,\mu(E_{k})\text{ \quad as }k\rightarrow
\infty\text{.} \label{Eq_cxbvnycvbyxcvcvcxyvcy}%
\end{equation}
It is a standard fact that the induced system $(Y,T_{Y},\xi_{Y})$ is a
Folklore map with infinitely many cylinders $W_{j}=Y\cap\{\varphi_{Y}=j\}$,
$j\geq1$. It is easy to see that $(E_{k}^{\prime})$ satisfies the assumptions
of Lemma \ref{T_MyCompactnessForExpoLimit} for this induced system. Indeed,
$E_{k}^{\prime}$ is (mod $\lambda$) an interval of the form $(c,c+\delta_{k}%
)$, and $\lambda(F_{k})=O(\lambda(E_{k}^{\blacktriangle}))$ follows from
$\lambda(W_{j})\sim T^{\prime}(0^{+})\,\lambda(W_{j+1})$ as $j\rightarrow
\infty$. Hence
\[
\mu(E_{k}^{\prime})\,\varphi_{E_{k}^{\prime}}^{Y}\overset{\mu_{Y}%
}{\Longrightarrow}\mathcal{E}\text{ \quad as }k\rightarrow\infty\text{.}%
\]
In view of Example \ref{Exple_BasicScenario}, Theorem
\ref{T_InduceHittingTimeStatsNew} applies. Combined with
(\ref{Eq_cxbvnycvbyxcvcvcxyvcy}) it gives (\ref{Eq_TheGoodExceptionalLimit}).
\end{proof}

\vspace{0.3cm}

\begin{remark}
[\textbf{Reformulation in terms of return-times}]In view of the general
duality between hitting-time statistics and return-time statistics established
in \cite{HaydnLacVai05}, (\ref{Eq_TheGoodExceptionalLimit}) is equivalent to
\begin{equation}
\mu(E_{k})\,\varphi_{E_{k}}\overset{\mu_{E_{k}}}{\Longrightarrow}%
\widetilde{\mathsf{R}}:=\Theta\cdot\theta^{-1}\mathcal{E}\text{ \quad as
}k\rightarrow\infty\text{,} \label{Eq_ncvsdbsdkbvab}%
\end{equation}
where the random variable $\Theta$ with $\Pr[\Theta=1]=1-\Pr[\Theta=0]=\theta$
is independent of $\mathcal{E}$. This is easily understood because a
subinterval $E_{k}\cap T^{-1}E_{k}$ of length $\lambda(E_{k}\cap T^{-1}%
E_{k})\sim\lambda(E_{k})/T^{\prime}(0^{+})=(1-\theta)\lambda(E_{k})$ re-enters
$E_{k}$ at once, which accounts for the atomic part $\Pr[\widetilde
{\mathsf{R}}=0]=1-\theta$ of $\widetilde{\mathsf{R}}$. The relations
(\ref{Eq_TheGoodExceptionalLimit}) and (\ref{Eq_ncvsdbsdkbvab}) can also be
rephrased in the language of extreme value statistics, see e.g. \cite{FFT}. In
that context, $\theta$ is called the \emph{extremal index}.
\end{remark}

\vspace{0.3cm}

\noindent\textbf{Maps with exceptionally unexceptional behaviour at a fixed
point.} Theorem \ref{T_TheWellKnownFolkloreFixedPoint} can be generalized to
repelling hyperbolic periodic points $x^{\ast}$ for other families of
interesting systems. In particular, it has been shown in \cite{Kell2012},
\cite{FFT} that in the context of uniformly expanding interval maps it
suffices to assume that $(X,T,\xi)$ is a \emph{Rychlik map} (that is, belongs
to the class studied in \cite{Ry}) with the \emph{additional assumptions} that
$T$ should be piecewise $\mathcal{C}^{1+\varepsilon}$ and that its invariant
density $h=d\mu/d\lambda$ should be bounded away from zero near $x^{\ast}$.

We now show that the latter condition cannot be dropped, even if the map has
very nice properties otherwise. To this end, we construct simple examples in
which the inducing principle of Theorem \ref{T_InduceHittingTimeStatsNew}
allows us to show that, in contrast to Example
\ref{T_TheWellKnownFolkloreFixedPoint}, \emph{neighbourhoods of hyperbolic
repelling fixed points of general Rychlik maps may still exhibit standard
exponential hitting time statistics}.

\begin{example}
[\textbf{Maps with standard HTS at uniform repellers}]%
\label{T_ExistenceOfRychlikMapsWithFunnyHTS1}There exist uniformly expanding
piecewise affine ergodic Rychlik maps $(X,T,\xi)$, which admit a fixed point
$x^{\ast}$ and neighbourhoods $E_{k}$ of $x^{\ast}$ satisfying $\lambda
(E_{k})\rightarrow0$ and
\begin{equation}
\mu(E_{k})\,\varphi_{E_{k}}\overset{\mu}{\Longrightarrow}\mathcal{E}\text{
\quad as }k\rightarrow\infty\text{,}%
\end{equation}
where $\mu$ denotes the unique absolutely continuous invariant probability measure.
\end{example}

\begin{proof}
\textbf{(i)} \emph{Structure of the map.} We are going to define a family of
uniformly expanding piecewise affine maps with big images on $X:=[0,1)$. In
particular, each of them is a Rychlik system $(X,T,\xi)$. They will be ergodic
w.r.t. Lebesgue measure $\lambda$, with unique right-continuous invariant
probability density $h=d\mu/d\lambda$ strictly positive on $(0,1) $, but with
$h(x)\rightarrow0$ as $x\searrow x^{\ast}:=0$. The fixed point $x^{\ast}$ is a
hyperbolic repeller, $T^{\prime}x^{\ast}=2$, but the sequence of cylinders
$E_{k}:=\xi_{k}(x^{\ast})\in\xi_{k}$ shrinking to this fixed point still
satisfies
\begin{equation}
\mu(E_{k})\,\varphi_{E_{k}}\overset{\mu}{\Longrightarrow}\mathcal{E}\text{
}\quad\text{as }k\rightarrow\infty\text{.} \label{Eq_GoalInExple007}%
\end{equation}

The basic partition takes the form $\xi=\{Z_{0},Z_{1},\ldots\}$ with
$Z_{j}=[z_{j},z_{j+1})$ for points $0=z_{0}<1/2=z_{1}<z_{2}<\ldots
<z_{j}\nearrow1$. The sequence $(z_{j})_{j\geq0}$, or equivalently the
sequence $(\lambda_{j})_{j\geq0}$ of lengths $\lambda_{j}:=z_{j+1}-z_{j}$,
will serve as a parameter which completely determines the system. For $x\in
Z_{0}=[0,1/2)$ set $Tx:=2x$, and let $Y:=Y_{0}:=Z_{0}^{c}=[1/2,1)$. Then, for
$i\geq1$, we see that $Y_{i}:=[2^{-(i+1)},2^{-i})=Y^{c}\cap\{\varphi_{Y}=i\}$.
Note also that $E_{k}:=\xi_{k}(x^{\ast})={\textstyle\bigcup\nolimits_{i\geq
k}} Y_{i}$ for $k\geq1$. On $Z_{j}$, $j\geq1$, we define $T$ to be decreasing
and affine, mapping $Z_{j}$ onto ${\textstyle\bigcup\nolimits_{i<j}}
Y_{i}\supseteq Y$, so that $T$ has slope $-s_{j}$ on $Z_{j}$ where
$s_{j}:=(1-2^{-j})/\lambda_{j}\sim\lambda_{j}^{-1}$ as $j\rightarrow\infty$.

For any $(z_{j})$ this gives a system $(X,T,\xi)$ for which $Y$ is a sweep-out
set (meaning that $X={\textstyle\bigcup\nolimits_{n\geq0}}T^{-n}Y$ (mod $\mu
$)). $T_{Y}$ is a pcw onto and pcw affine map, hence Folklore (and ergodic)
with invariant measure $\lambda_{Y}$ on $Y$. By standard arguments the
original map $T$ is therefore ergodic on $X$ w.r.t. $\lambda$. Also being a
Rychlik map, $T$ has a unique invariant probability density with a
right-continuous version $h$ of bounded variation. Note that $T$ is piecewise
affine and Markov for the finer partition $\xi^{\prime}:=\{\ldots Y_{2}%
,Y_{1},Z_{1},Z_{2},\ldots\}$. Therefore $h$ is constant on each element of the
partition $\{Y_{j}\}_{j\geq0}$ generated by the image sets $TZ^{\prime}$,
$Z^{\prime}\in\xi^{\prime}$. Hence, $h=\sum_{j\geq0}\eta_{j}1_{Y_{j}}$ with
$\eta_{j}=\mu(Y_{j})/\lambda(Y_{j})=2^{j+1}\mu(Y_{j})$. Below we give an
explicit description of the invariant measure $\mu$ in terms of $(\lambda
_{j})_{j\geq0}$, and show in particular that $h$ is strictly positive on
$(0,1)$.\newline\newline\textbf{(ii)} \emph{The invariant measure.} For a
probability $\mu$ with density of the form $h=\sum_{j\geq0}\eta_{j}1_{Y_{j}}$
to be $T$-invariant, we must first have
\[
\mu(Z_{1})=\mu(Y_{1}\cap T^{-1}Z_{1})+{\textstyle\sum\nolimits_{j\geq1}}%
\mu(Z_{j}\cap T^{-1}Z_{1})\text{.}%
\]
Since $Y_{1}$ is mapped onto $Y_{0}$ without distortion, $\mu(Y_{1}\cap
T^{-1}Z_{1})=\mu(Y_{1})\lambda(Z_{1})/\lambda(Y_{0})=\eta_{1}\lambda_{1}/2$,
and as $Z_{j}$ is mapped onto ${\textstyle\bigcup\nolimits_{i<j}}Y_{i}$
without distortion, $\mu(Z_{j}\cap T^{-1}Z_{1})=\mu(Z_{j})\lambda
(Z_{1})/\lambda({\textstyle\bigcup\nolimits_{i<j}}Y_{i})=\eta_{0}\lambda
_{1}s_{j}^{-1}$. Letting $\sigma(m):=\eta_{0}2^{-(m+1)}\sum_{i>m}s_{i}^{-1} $,
$m\geq0$, this leads to $\mu(Y_{0})=\mu(Y_{1})+\sigma(0)$.

Likewise, for every $j\geq1$, $\mu(Y_{j})=\mu(Y_{j+1})+{\textstyle\sum
\nolimits_{i>j}} \mu(Z_{i}\cap T^{-1}Y_{j})$, and using the definition of the
individual branches this becomes $\mu(Y_{j})=\mu(Y_{j+1})+\sigma(j)$.

We therefore see that
\begin{equation}
\mu(Y_{j})={\textstyle\sum\nolimits_{i\geq j}} \sigma(i)\text{ \quad for
}j\geq0\text{.} \label{Eq_CharacterizeMuInExple}%
\end{equation}
\textbf{(iii)} \emph{Hitting-time statistics.} We first consider the induced
map $T_{Y}$ and $E_{k}^{\prime}:=Y\cap T^{-1}E_{k}$, $k\geq1$. The latter
defines asymptotically rare events in $Y$, and it is easily seen that each
$E_{k}^{\prime}$ is a union of cylinders from $\xi_{Y}$ (recall that $T$ is
Markov for $\xi^{\prime}$). Consequently (see Lemma
\ref{T_MyCompactnessForExpoLimit} and Remark \ref{R_UnionOfCyls}),
\begin{equation}
\mu(E_{k}^{\prime})\,\varphi_{E_{k}^{\prime}}^{Y}\overset{\mu_{Y}%
}{\Longrightarrow}\mathcal{E}\text{ }\quad\text{as }k\rightarrow\infty\text{.}%
\end{equation}
Due to Theorem \ref{T_InduceHittingTimeStatsNew}, then
\[
\mu(E_{k}^{\prime})\,\varphi_{E_{k}}\overset{\mu}{\Longrightarrow}%
\mathcal{E}\text{ }\quad\text{as }k\rightarrow\infty\text{,}%
\]
and our claim (\ref{Eq_GoalInExple007}) follows in case $\mu(E_{k}^{\prime
})\sim\mu(E_{k})$ as $k\rightarrow\infty$. However, we are exactly in the
situation of Remark \ref{Rem_cfcfcfccf}, meaning that the latter is fulfilled
iff%
\begin{equation}
{\textstyle\sum\nolimits_{j>k}} \mu(Y_{j})=o(\mu(Y_{k}))\text{ \quad as
}k\rightarrow\infty\text{,} \label{Eq_hjdsfgkajgfkjfg65478151}%
\end{equation}
because $E_{k}={\textstyle\bigcup\nolimits_{j\geq k}} Y_{j}$. In view of
(\ref{Eq_CharacterizeMuInExple}) and the definition of $\sigma(m)$ this holds
if $\lambda_{j}\searrow0$ sufficiently fast.

To get specific examples, note that if the $\lambda_{i}$ are such that
$\lambda_{i+1}\leq2^{-(i+3)}\lambda_{i}$ for $i\geq k_{0}$, then (using
$\lambda_{i}/2\leq s_{i}^{-1}\leq\lambda_{i}$) we find that $\mu(Y_{k+1}%
)\leq2^{-(k+1)}\mu(Y_{k})$ for $k\geq k_{0}$, which implies
(\ref{Eq_hjdsfgkajgfkjfg65478151}).
\end{proof}

\vspace{0.3cm}

\begin{remark}
In the particular version of the maps we describe in the proof, the $E_{k}$
are the \emph{one-sided} neighbourhoods $[0,2^{-k})$ of $x^{\ast}$. It is also
easy to construct a variant in which the $E_{k}$ are symmetric
\emph{two-sided} neighbourhoods of a repeller which lies in the center of some cylinder.
\end{remark}

\vspace{0.3cm}

\section{A limit theorem for indifferent fixed points}

\noindent\textbf{Previous results.} In a recent paper \cite{FFTV} the authors
have studied hitting-time limits for a concrete parametrized family of
non-uniformly expanding interval maps $T_{p}:[0,1]\rightarrow\lbrack0,1]$
possessing an indifferent fixed point at $x^{\ast}=0$,
\begin{equation}
T_{p}x:=\left\{
\begin{array}
[c]{cc}%
x+2^{p}x^{1+p} & \text{for }x<1/2\text{,}\\
2x-1 & \text{for }x>1/2\text{,}%
\end{array}
\right.  \label{Eq_ZeMaps}%
\end{equation}
with parameter $p\in(0,1)$. In this parameter range, each $T_{p}$ posesses a
unique absolutely continuous (w.r.t. Lebesgue measure $\lambda$) invariant
probability measure $\mu_{p}$ with density $h_{p}$ strictly positive and
continuous on $(0,1]$. Among other things, the hitting-time distributions of
small neighbourhoods $[0,\epsilon]$ of the distinguished point $x^{\ast}=0$
were analysed and shown to converge to a standard exponential random variable
$\mathcal{E}$, but under a normalization essentially different from the usual
factor $\mu_{p}([0,\epsilon])\approx\epsilon^{1-p}$ as $\epsilon\searrow0$.
More precisely, in \cite{FFTV} this result was established, using extreme
value theory, only under the assumption that $p\in(0,\sqrt{5}-2)$. The authors
state the conjecture that the same assertion should be true at least for
$p\in(0,1/2)$.

In the present section we clarify the asymptotics of the hitting-time
distributions of these exceptional\ families of rare events by extending the
result of \cite{FFTV} to arbitrary $p\in(0,1)$ and, in fact, to a more general
class of maps. \emph{We emphasize that no assumption on the analytical
behaviour at the indifferent point akin to regular variation is needed, and
that the argument does not use information on the decay of correlations for}
$T_{p}$. Instead, we are going to use Theorem
\ref{T_InduceHittingTimeStatsNew} in a straightforward manner. The
significantly more technical results from extreme value theory employed in
\cite{FFTV} are not required.

\begin{remark}
[\textbf{The infinite measure case,} $p\geq1$]Asymptotic hitting-time
distributions of neighbourhoods of neutral fixed points for \emph{infinite
measure preserving} situations (as encountered when $T=T_{p}$ as in
(\ref{Eq_ZeMaps}) but with parameter $p\geq1$) have been obtained in \cite{Z8}.

For cylinders shrinking to \emph{typical} points of such a null-recurrent map,
hitting-time statistics have been clarified more recently in \cite{PSZ2}, see
also \cite{RZ}.
\end{remark}

\vspace{0.3cm}

\noindent\textbf{The behaviour of an intermittent map at the neutral source.}
Again we formulate the result in the setup of simple maps with two full
branches. It can be extended to more general maps and periodic points via
routine arguments.

\begin{theorem}
[\textbf{HTS for neutral fixed points of simple maps}]\label{T_HTSatNeutralFP}%
Let $(X,T,\xi)$ be piecewise increasing with $X=[0,1] $ and $\xi
=\{(0,c),(c,1)\}$, mapping each $Z\in\xi$ onto $(0,1)$. Assume that
$T\mid_{(c,1)}$ admits a uniformly expanding $\mathcal{C}^{2}$ extension to
$[c,1]$, while $T\mid_{(0,c)}$ extends to a $\mathcal{C}^{2}$ map on $(0,c]$
and is expanding except for an indifferent fixed point at $x^{\ast}=0$: for
every $\varepsilon>0$ there is some $\rho(\varepsilon)>1 $ such that
$T^{\prime}\geq\rho(\varepsilon)$ on $[\varepsilon,c]$, while $T0=0 $ and
$\lim\nolimits_{x\searrow0}T^{\prime}x=1$ with $T^{\prime}$ increasing on some
$(0,\delta)$. Suppose also that
\begin{gather}
\text{there is a continuous decreasing function }g\text{ on }(0,c]\text{
with}\nonumber\\
\int_{0}^{c}g(x)\,dx<\infty\text{ \quad and \quad}\left\vert T^{\prime\prime
}\right\vert \leq g\text{ on }(0,c]\text{.} \label{Eq_MaxMax}%
\end{gather}
Let $(E_{k})_{k\geq1}$ be a sequence of intervals which contain the fixed
point $x^{\ast}$, and such that $\lambda(E_{k})\rightarrow0$. Then,
\begin{equation}
\mu(E_{k})\,\varphi_{E_{k}}\overset{\mu}{\Longrightarrow}\infty\quad\text{as
}k\rightarrow\infty\text{,} \label{Eq_bncvhsjfgsv}%
\end{equation}
where $\mu\ll\lambda$ is the unique invariant probability measure of $T$. Much
more precisely,%
\begin{equation}
\tfrac{h(c)}{T^{\prime}(c^{+})}\,\lambda(E_{k})\,\varphi_{E_{k}}\overset{\mu
}{\Longrightarrow}\mathcal{E}\text{ \quad as }k\rightarrow\infty\text{,}
\label{Eq_hdgfjashdfjhsadfhsdafghj}%
\end{equation}
where $h$ is the continuous version of the invariant density $d\mu/d\lambda$.
\end{theorem}

\begin{proof}
Let $E_{k}^{\prime}:=Y\cap T^{-1}E_{k}$ with $Y:=(c,1)\in\xi$ the right-hand
cylinder. Then
\begin{equation}
T_{Y}\text{ is a Folklore map.} \label{Eq_InduceFolkloresdkjhfka}%
\end{equation}
Specifically, with $c_{0}:=0$, $c_{1}:=c$, and $c_{j+1}$ such that
$c_{j+1}<Tc_{j+1}=c_{j}$ for $j\geq1$, the sets $V_{j}:=(c_{j},c_{j-1})$
accumulate at $x^{\ast}$ and satisfy $\lambda(V_{j+1})\sim\lambda(V_{j})$
since $V_{j}=TV_{j+1}$ and $\lim\nolimits_{x\searrow x^{\ast}}T^{\prime}x=1 $.
In particular, $\lambda(V_{j})=o(\lambda(%
%TCIMACRO{\tbigcup \nolimits_{i>j}}%
%BeginExpansion
{\textstyle\bigcup\nolimits_{i>j}}
%EndExpansion
V_{i}))$ as $j\rightarrow\infty$. Since $V_{1}=Y$, the cylinders of $T_{Y}$
are the sets $W_{j}:=Y\cap T^{-1}V_{j-1}=Y\cap\{\varphi_{Y}=j\}$, $j\geq1$,
which accumulate at $c^{+}$ and satisfy $\lambda(W_{j})\sim\lambda
(V_{j})/T^{\prime}c^{+}$. Hence,
\begin{equation}
\lambda(W_{j})=o(\lambda(%
%TCIMACRO{\tbigcup \nolimits_{i>j}}%
%BeginExpansion
{\textstyle\bigcup\nolimits_{i>j}}
%EndExpansion
W_{i}))\text{ \quad as }j\rightarrow\infty\text{.} \label{Eq_fadxcfadss}%
\end{equation}
Adler's condition for $T_{Y}$ follows from assumption (\ref{Eq_MaxMax}) by an
analytic argument which goes back to \cite{T1}, see for example \S 3 of
\cite{Z3} or \S 4 of \cite{T3}.

As observed in Example \ref{Exple_BasicScenario}, points of $Y$ can enter
$E_{k}$ only via $E_{k}^{\prime}$, and Lemma \ref{T_MyCompactnessForExpoLimit}
for the induced system yields $\mu_{Y}(E_{k}^{\prime})\,\varphi_{E_{k}%
^{\prime}}^{Y}\overset{\mu_{Y}}{\Longrightarrow}\mathcal{E}$. Indeed,
$E_{k}^{\prime}$ is a one-sided neighborhood of $c$, so that (up to endpoints
of cylinders) $%
%TCIMACRO{\tbigcup \nolimits_{i>j(k)}}%
%BeginExpansion
{\textstyle\bigcup\nolimits_{i>j(k)}}
%EndExpansion
W_{i}\subseteq E_{k}^{\prime}\subseteq%
%TCIMACRO{\tbigcup \nolimits_{i\geq j(k)}}%
%BeginExpansion
{\textstyle\bigcup\nolimits_{i\geq j(k)}}
%EndExpansion
W_{i}$ for suitable $j(k)\rightarrow\infty$. Consequently, using the notation
of the lemma and (\ref{Eq_fadxcfadss}) above, $E_{k}^{\blacktriangle}=%
%TCIMACRO{\tbigcup \nolimits_{i>j(k)}}%
%BeginExpansion
{\textstyle\bigcup\nolimits_{i>j(k)}}
%EndExpansion
W_{i}$ and $F_{k}=W_{j(k)}$ satisfy $\lambda(F_{k})=o(\lambda(E_{k}%
^{\blacktriangle})) $.
  
Now Theorem \ref{T_InduceHittingTimeStatsNew} shows that
\begin{equation}
\mu(E_{k}^{\prime})\,\varphi_{E_{k}}\overset{\mu}{\Longrightarrow}%
\mathcal{E}\quad\text{as }k\rightarrow\infty. \label{Eq_cbvcbxvxbcvymvffffff}%
\end{equation}
Straightforward calculation, using continuity of $h$ on $Y$, shows that
\[
\mu(E_{k}^{\prime})\sim\frac{h(c)}{T^{\prime}(c^{+})}\lambda(E_{k})\text{
\quad as }k\rightarrow\infty\text{,}%
\]
proving (\ref{Eq_hdgfjashdfjhsadfhsdafghj}). But then (\ref{Eq_bncvhsjfgsv})
follows at once, since in the present situation, $\mu(E_{k}^{\prime}%
)=\mu(E_{k})-\mu(Y^{c}\cap T^{-1}E_{k})$ satisfies $\mu(E_{k}^{\prime}%
)=o(\mu(E_{k}))$. This is because $\mu(Y^{c}\cap T^{-1}E_{k})\sim\mu
(E_{k})/T^{\prime}(0^{+})\sim\mu(E_{k})$ as $k\rightarrow\infty$.
\end{proof}

\vspace{0.3cm}

\begin{remark}
Under the assumptions of the theorem, the return-time distributions (law of
$\varphi_{E_{k}}$ under $\mu_{E_{k}}$) are of limited interest, since
$\mu_{E_{k}}(\varphi_{E_{k}}=1)\rightarrow1$.
\end{remark}

\begin{remark}
In the statement of the theorem, condition (\ref{Eq_MaxMax}) can be replaced
by its consequence (\ref{Eq_InduceFolkloresdkjhfka}) because the proof only
depends on the latter property.
\end{remark}

\begin{example}
Each map $T=T_{p}$ from (\ref{Eq_ZeMaps}) with $p\in(0,1)$ trivially satisfies
(\ref{Eq_MaxMax}). Therefore, whenever $E_{k}=[0,\epsilon_{k}]$ with
$\epsilon_{k}\rightarrow0$, the theorem shows that
\begin{equation}
\tfrac{h_{p}(\frac{1}{2})}{2}\,\epsilon_{k}\,\varphi_{E_{k}}\overset{\mu
}{\Longrightarrow}\mathcal{E}\text{ \quad as }k\rightarrow\infty\text{.}
\label{Eq_FinalExpleSeq}%
\end{equation}
Consequently, we also have (generalizing \cite{FFTV})
\begin{equation}
\tfrac{h_{p}(\frac{1}{2})}{2}\,\epsilon\,\varphi_{\lbrack0,\epsilon]}%
\overset{\mu}{\Longrightarrow}\mathcal{E}\text{ \quad as }\epsilon
\searrow0\text{.} \label{Eq_FinalExpleCts}%
\end{equation}
(Otherwise there would be some sequence $(\epsilon_{k})$ violating
(\ref{Eq_FinalExpleSeq}).)
\end{example}

\end{document}